\documentclass[12pt]{amsart}
\usepackage[margin=1in]{geometry}
\usepackage{mathptmx}
\usepackage{mathtools}
\usepackage{amsthm}
\usepackage{amssymb}
\usepackage{amsrefs}
\usepackage{graphicx}
\usepackage{tikz}
	\usetikzlibrary{decorations.pathreplacing}
	\usetikzlibrary{patterns}
\usepackage{caption}

\newcommand{\ra}{\rightarrow}
\newcommand{\pr}{\prime}
\newcommand{\de}{\partial}

\newcommand{\R}{\mathbb{R}}

\DeclareMathOperator{\id}{id}

\DeclareMathAlphabet{\mathcal}{OMS}{cmsy}{m}{n}

\newcommand{\indep}{\rotatebox[origin=c]{90}{$\models$}}

\newcommand{\hookuparrow}{\mathrel{\rotatebox[origin=c]{90}{$\hookrightarrow$}}}

\newtheorem{thm}{Theorem}
\newtheorem*{thm*}{Relative Version}

\theoremstyle{definition}

\newtheorem*{defin*}{Definition}
\newtheorem*{observation}{Observation}
\theoremstyle{plain}

\newtheorem{proposition}{Proposition}
\newtheorem{remark}{Remark}

\begin{document}

\title{Approximation rigidity and $h$-principle for Bing Spines}

\author{Michael Freedman}
\address{\hskip-\parindent
	Michael Freedman \\
    Microsoft Research, Station Q, and Department of Mathematics \\
    University of California, Santa Barbara \\
    Santa Barbara, CA 93106 \\
}
\author{T. T$\hat{\mathrm{a}}$m Nguy$\tilde{\hat{\mathrm{e}}}$n-Phan}
\address{\hskip-\parindent
	Tam Nguyen-Phan \\
	Max Planck Institute for Mathematics \\
	Vivatgasse 7 \\
	Bonn, 53111 \\
	Germany \\
}

\begin{abstract}
	We show that all PL manifolds of dimension $\geq 3$ have spines similar to Bing's house with two rooms. Beyond this we explore approximation rigidity and an $h$-principle.
\end{abstract}

\maketitle

\section{Introduction}

Let $M^n$ be a PL manifold of dimension $n$ with nonempty boundary $\partial M$. We say that a subset $S \subset M$ of dimension $<n$ is a \emph{spine} if there is a PL compatible triangulation $\Delta$ of $M$ and a sequence of elementary collapses $c_1, \dots, c_k$, the composition $c$ of which is a collapse from $M$ to the subcomplex $S$. Then $M$ has the structure of a mapping cylinder given by $c \big\vert_{\de M}: \de M \ra S$. Usually the map $c \big\vert_{\de M}$ is not immersive (e.g. when $M$ is a closed ball and $S$ is a point); for this reason we would like introduce the notion of a \emph{Bing spine}.

\subsection*{Bing spine} A spine $S$ is called a \emph{Bing spine} if the compound collapse $c$ restricts to a locally flat PL immersion $c \big\vert_{\de M}: \de M \ra S$. The basic example of a Bing spine is Bing's house with two rooms, which is a Bing spine for the 3-ball $B^3$.

\begin{figure}[!ht]
	\centering
	\begin{tikzpicture}
		\draw (0,0) -- (0,4);
		\draw (4,0) -- (4,4);
		\draw (2,4) ellipse (2 and 0.25);
		\draw (2,0) ellipse (2 and 0.25);
		\draw (2,2) ellipse (2 and 0.25);
		\draw (1,2) ellipse (0.2 and 0.1);
		\draw (1,4) ellipse (0.2 and 0.1);
		\draw (1.2, 4) -- (1.2, 2);
		\draw (0.8, 4) -- (0.8, 2);
		\draw (3,0) ellipse (0.2 and 0.1);
		\draw (3,2) ellipse (0.2 and 0.1);
		\draw (3.2, 2) -- (3.2, 0);
		\draw (2.8, 2) -- (2.8, 0);
		\path [pattern = horizontal lines, pattern color = gray] (3.2,2.1) rectangle (4,0);
		\path [pattern = horizontal lines, pattern color = gray] (0.8,2) rectangle (0,4.05);
	\end{tikzpicture}
	\caption{The classical Bing house.}
	\label{2d-bing}
\end{figure}
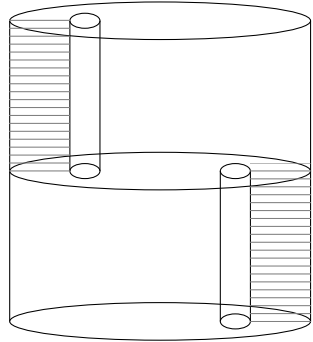

One way to arrive at the above picture is to take a solid 3-ball $M = B^3$ and press it in with two hands, one from the top and the other from the bottom, and then make each hand grab the wrist of the other before ``inflating" both hands until the deformed ball becomes the 2-complex $S$ in Figure \ref{2d-bing}. It should be clear from this process that $S$ is a Bing spine of $M$. 

Another example of a Bing spine is the last picture in Figure \ref{basic-model}, which illustrates how the process described above is carried out with six hands to give a Bing spine of the solid torus.   

The main result of this paper is that Bing spines (BS) are rather common. We will use the Hausdorff topology on subsets of $M$ to measure the ``distance" between various spines.

\begin{thm}[Approximation Theorem]\label{approximation}
	$n \geq 3$, all compact PL $n$-manifolds $M$ with $\de M \neq \varnothing$ admit a Bing spine. Furthermore, Bing spines are dense in the Hausdorff topology on the space of spines.
\end{thm}

\begin{thm*}
	Let $M^n$, $n>3$, be a compact PL manifold with $A \subset \de M$ a compact codimension 0 submanifold of $\de M$. Suppose $BS^{n-2}$ in a Bing spine for $A$, then any spine $S$ for $M$ which meets $A$ in $BS^{n-2}$ can be approximated by a relative Bing spine $BS^{n-1}$; $A \cap BS^{n-1} = BS^{n-2}$ and $BS^{n-1}$ is parameterized by a locally flat PL immersion $\overline{\de M \setminus A} \looparrowright BS^{n-1}$, extending the locally flat PL immersion $\de A \looparrowright BS^{n-2}$ parameterizing $BS^{n-2}$. Again approximation is in the Hausdorff topology.
\end{thm*}

A codimension one PL immersion $M^{n-1} \looparrowright X^n$ of PL manifolds is locally flat if near any point $p$ in the source the local embedding near $p$ extends to a bi-collar of a neighborhood $\mathcal{N}_p$ of p:

\begin{figure}[ht]
	\centering
	\begin{tikzpicture}
		\node at (0,0) {$\mathcal{N}_p \hookrightarrow \mathcal{N}(p) \times [-1,1]$};
		\node at (-1.6,0.6) {$\hookuparrow$};
		\node at (1.4,0.6) {$\hookuparrow$};
		\node at (-1.6,1.15) {$M^{n-1}$};
		\node at (1.4,1.15) {$X^n$};
		\draw[->] (-0.85,0.95) to[out=90,in=0] (-1,1.35) to[out=180,in=90] (-1.1,1.25) to[out=-90,in=180] (-0.85,1.15) -- (1,1.15);
		\node at (0,-0.5) {$q \mapsto q \times 0$};
		\node at (4.5,0.5) {all maps in PL category};
		\node at (-5,0) {$\hspace{1em}$};
	\end{tikzpicture}
\end{figure}

\noindent
\textbf{An $h$-principle} would say that for suitable topologies the space S of spines deformation retracts to the space BS of Bing spines. We prove a concordance based, semi-simplicial, $h$-principle of the form below, stated in detail in Section \ref{weak h-principle}.

\begin{thm}[rough statement]\label{roughthm}
	Assume $(M^n, \de M)$, $n \geq 3$ and $\de M \neq \varnothing$, then in the concordance-based, semi-simplicial setting, for any $k$-parameter family $\mathcal{F}$ of PL collapses $c_f: M \ra S_f$, $f \in \mathcal{F}$, and an $\epsilon>0$, there is an $\epsilon$-approximation $c_f^\pr:M \ra BS_f$ so that $BS_f$ are Bing spines with $d^{\text{Hausdorff}}(S_f,BS_f) < \epsilon$ and $d^{\text{sup norm}}(c_f, c_f^\pr) < \epsilon$, $f \in \mathcal{F}$.
\end{thm}

This theorem is for example similar to the main theorem of immersion theory \cite{hirsch} that any finite dimensional family of maps (of codimension $\geq 1$) can be perturbed to a family of immersions iff the family is covered by a family of tangent bundle injections. In our case, we work with manifold boundaries whose normal collar supplies the tangential information so our statement is even simpler.

Actually both theorems depend on $h$-principle thinking. Just as the Ur $h$-theorem, the Whitney-Gorestein Theorem \cite{whitney} introduces a basic local model, ``wiggle-wiggle,'' to ``soften" immersions of a circle into the plane. (Before lifting a homotopy between immersions of equal tangential degrees to a \emph{regular} homotopy, a series of wiggles is introduced to ``soften'' the immersion.):

\begin{figure}[!ht]
	\begin{tikzpicture}[scale=0.7]
		\draw (-0.2,-1) [out=80,in=190] to (2.2,4.2);
		\draw [->] (3,1.5) -- (5.5,1.5);
		\draw (7.5,-1.1) to [out=80,in=180] (8.7,0.7) to [out=0,in=90] (9.2,0.2) to [out=-90,in=0] (8.6,-0.2) to [out=180,in=0] (6.9,1.3) to [out=180, in=90] (6.4,0.9) to [out=-90, in=180] (7.1,0.4);
		\draw (7.1,0.4) to [out=0, in=180] (9.3,2.1) to [out=0,in=90] (9.8,1.6) to [out=-90,in=0] (9.2,1.2) to [out=180,in=0] (7.5,2.7) to [out=180, in=90] (7,2.3) to [out=-90, in=180] (7.7,1.8);
		\draw (7.7,1.8) to [out=0, in=180] (9.9,3.5) to [out=0,in=90] (10.4,3) to [out=-90,in=0] (9.8,2.6) to [out=180,in=0] (8.1,4.1) to [out=180, in=90] (7.6,3.7) to [out=-90, in=180] (8.3,3.2);
		\draw (8.3,3.2) to [out=0, in=180] (9.3,4.6);
	\end{tikzpicture}
	\centering
	\caption{}
\end{figure}

To prove both Theorems \ref{approximation} and \ref{roughthm} we introduce a basic local model $BS_{n,\epsilon}$ (one for each dimension $n \geq 3$ and each $\epsilon>0$) to make PL immersions sufficiently flexible that the immersive property survives each elementary collapse $c_i$. The model is a family of Bing spines $BS_{n,\epsilon}$ for $S^{n-2} \times D^2$ which admit a PL immersion $c_{n,\epsilon}: S^{n-2} \times S^1 \ra BS_{n,\epsilon}$ within $\epsilon$ of the projection $S^{n-2} \times S^1 \ra S^{n-2}$, $\epsilon > 0$, $n \geq 3$. The composite collapse $c_{n,\epsilon}$ (as is always the case with a composition of elementary collapses) defines a mapping cylinder structure (MC) on the total space, $S^{n-2} \times D^2$, with source $S^{n-2} \times S^1$ and target $BS_{n,\epsilon}$. These models will be inserted around the boundary of a free face to modify the current immersion before that face is collapsed. We call these local models ``Bing ruffs,'' to evoke the neck-ware popular among Elizabethan royalty. This modification allows the then current immersion to survive the next collapse, as we will see in the proof of Theorem \ref{approximation} in the next section.

\begin{figure}[!ht]
	\centering
	\begin{tikzpicture}
		\node at (0,4) {$\cong$};
\draw (2,4) circle (1);
\node at (2.1,4.1) {$B^{n-1}$};
\node at (3.4,4.8) {$S^{n-2}$};
\draw (-2,3.1) -- (-3,5.1) -- (-4,3.1) -- cycle;
\draw [color=gray] (-1.4,3.4) -- (-2,3.1) -- (-1.8,2.5);
\draw [color=gray] (-2,3.1) -- (-1.4,2.9);
\draw [color=gray] (-4.6,3.4) -- (-4,3.1) -- (-4.2,2.5);
\draw [color=gray] (-4,3.1) -- (-4.6,2.9);
\draw [color=gray] (-2.6,5.4) -- (-3,5.1) -- (-3.4,5.4);

\node at (-3,4) {$\Delta_{n-1}$};
\node at (-3,3.5) {free face};
\node at (-1.4,4.9) {$\partial(\Delta_{n-1}) = S^{n-2}$};
\draw [->] (-1.4, 4.6) -- (-2.4,4.1);

\draw (0.5,0.9) .. controls (0.3,0.4) and (0.9,0.3) .. (1.4,-0.2) .. controls (1.6,-0.4) .. (1.73,-0.63);
\draw (0.5,0.9) to [out=70,in=150] (0.65,0.95) to [out=-20,in=80] (0.67,0.8) .. controls (0.55,0.45) and (1.1,0.25) .. (1.2,0.37) .. controls (1.3,0.45) and (1.2,0.55) .. (1.14,0.6);
\draw (1.1,0.85) ellipse (0.15 and 0.17);
\draw (1.23,0.76) .. controls (1.6,0.3) and (1.8,-0.2) .. (1.95,-0.45);
\draw (1.2,0.72) .. controls (1.1,0.5) and (1.0,0.4) .. (0.74,0.5);
\draw (0.95, 0.9) .. controls (0.9,0.8) and (0.8,0.75) .. (0.67,0.8);
\draw (1,1) to [out=135,in=45] (0.85,0.99) to [out=225,in=155] (0.89,0.83);

\draw [rotate around={120:(0,-1)}] (0.5,0.9) .. controls (0.3,0.4) and (0.9,0.3) .. (1.4,-0.2) .. controls (1.6,-0.4) .. (1.73,-0.63);
\draw [rotate around={120:(0,-1)}] (0.5,0.9) to [out=70,in=150] (0.65,0.95) to [out=-20,in=80] (0.67,0.8) .. controls (0.55,0.45) and (1.1,0.25) .. (1.2,0.37) .. controls (1.3,0.45) and (1.2,0.55) .. (1.14,0.6);
\draw [rotate around={120:(0,-1)}] (1.1,0.85) ellipse (0.15 and 0.17);
\draw [rotate around={120:(0,-1)}] (1.23,0.76) .. controls (1.6,0.3) and (1.8,-0.2) .. (1.95,-0.45);
\draw [rotate around={120:(0,-1)}] (1.2,0.72) .. controls (1.1,0.5) and (1.0,0.4) .. (0.74,0.5);
\draw [rotate around={120:(0,-1)}] (0.95, 0.9) .. controls (0.9,0.8) and (0.8,0.75) .. (0.67,0.8);
\draw [rotate around={120:(0,-1)}] (1,1) to [out=135,in=45] (0.85,0.99) to [out=225,in=155] (0.89,0.83);

\draw [rotate around={240:(0,-1)}] (0.5,0.9) .. controls (0.3,0.4) and (0.9,0.3) .. (1.4,-0.2) .. controls (1.6,-0.4) .. (1.73,-0.63);
\draw [rotate around={240:(0,-1)}] (0.5,0.9) to [out=70,in=150] (0.65,0.95) to [out=-20,in=80] (0.67,0.8) .. controls (0.55,0.45) and (1.1,0.25) .. (1.2,0.37) .. controls (1.3,0.45) and (1.2,0.55) .. (1.14,0.6);
\draw [rotate around={240:(0,-1)}] (1.1,0.85) ellipse (0.15 and 0.17);
\draw [rotate around={240:(0,-1)}] (1.23,0.76) .. controls (1.6,0.3) and (1.8,-0.2) .. (1.95,-0.45);
\draw [rotate around={240:(0,-1)}] (1.2,0.72) .. controls (1.1,0.5) and (1.0,0.4) .. (0.74,0.5);
\draw [rotate around={240:(0,-1)}] (0.95, 0.9) .. controls (0.9,0.8) and (0.8,0.75) .. (0.67,0.8);
\draw [rotate around={240:(0,-1)}] (1,1) to [out=135,in=45] (0.85,0.99) to [out=225,in=155] (0.89,0.83);

\draw (-1.3,0.55) .. controls (-1,0.2) and (-0.2,0.8) .. (0.4,0.6);
\draw (-1.3,0.55) to [out=145, in=225] (-1.3,0.7) to [out=45,in=145] (-1.15,0.7);
\draw (-1.15,0.7) .. controls (-0.85,0.5) and (-0.75,0.7) .. (-0.7,0.8) .. controls (-0.67,0.9) and (-0.75,0.95) .. (-0.86,0.93);
\draw (-1,1) ellipse (.15 and .17);
\draw (-0.86,0.93) .. controls (-0.95,0.72) and (-0.95,0.69) .. (-0.99,0.65);
\draw (-1.15, 1) .. controls (-1.22,0.8) .. (-1.24,0.75);
\draw (-0.88,1.1) .. controls (-0.2,1.2) and (0.2,0.9) .. (0.45,0.9);
\draw (-1.13,1.08) to [out=190,in=90] (-1.25,0.98) to [out=-90,in=180] (-1.18,0.9);

\draw [rotate around={120:(0,-1)}] (-1.3,0.55) .. controls (-1,0.2) and (-0.2,0.8) .. (0.4,0.6);
\draw [rotate around={120:(0,-1)}] (-1.3,0.55) to [out=145, in=225] (-1.3,0.7) to [out=45,in=145] (-1.15,0.7);
\draw [rotate around={120:(0,-1)}] (-1.15,0.7) .. controls (-0.85,0.5) and (-0.75,0.7) .. (-0.7,0.8) .. controls (-0.67,0.9) and (-0.75,0.95) .. (-0.86,0.93);
\draw [rotate around={120:(0,-1)}] (-1,1) ellipse (.15 and .17);
\draw [rotate around={120:(0,-1)}] (-0.86,0.93) .. controls (-0.95,0.72) and (-0.95,0.69) .. (-0.99,0.65);
\draw [rotate around={120:(0,-1)}] (-1.15, 1) .. controls (-1.22,0.8) .. (-1.24,0.75);
\draw [rotate around={120:(0,-1)}] (-0.88,1.1) .. controls (-0.2,1.2) and (0.2,0.9) .. (0.45,0.9);
\draw [rotate around={120:(0,-1)}] (-1.13,1.08) to [out=190,in=90] (-1.25,0.98) to [out=-90,in=180] (-1.18,0.9);

\draw [rotate around={240:(0,-1)}] (-1.3,0.55) .. controls (-1,0.2) and (-0.2,0.8) .. (0.4,0.6);
\draw [rotate around={240:(0,-1)}] (-1.3,0.55) to [out=145, in=225] (-1.3,0.7) to [out=45,in=145] (-1.15,0.7);
\draw [rotate around={240:(0,-1)}] (-1.15,0.7) .. controls (-0.85,0.5) and (-0.75,0.7) .. (-0.7,0.8) .. controls (-0.67,0.9) and (-0.75,0.95) .. (-0.86,0.93);
\draw [rotate around={240:(0,-1)}] (-1,1) ellipse (.15 and .17);
\draw [rotate around={240:(0,-1)}] (-0.86,0.93) .. controls (-0.95,0.72) and (-0.95,0.69) .. (-0.99,0.65);
\draw [rotate around={240:(0,-1)}] (-1.15, 1) .. controls (-1.22,0.8) .. (-1.24,0.75);
\draw [rotate around={240:(0,-1)}] (-0.88,1.1) .. controls (-0.2,1.2) and (0.2,0.9) .. (0.45,0.9);
\draw [rotate around={240:(0,-1)}] (-1.13,1.08) to [out=190,in=90] (-1.25,0.98) to [out=-90,in=180] (-1.18,0.9);

\draw [color=white, line width = 3pt] (0.56,0.98) arc (18:330:0.6 and 0.75);
\draw (0.56,0.98) arc (18:330:0.6 and 0.75);
\draw [rotate around={120:(0,-1)}, color=white, line width = 3pt] (0.56,0.98) arc (18:330:0.6 and 0.75);
\draw [rotate around={120:(0,-1)}] (0.56,0.98) arc (18:330:0.6 and 0.75);
\draw [rotate around={-120:(0,-1)},color=white, line width = 3pt] (0.56,0.98) arc (18:330:0.6 and 0.75);
\draw [rotate around={-120:(0,-1)}] (0.56,0.98) arc (18:330:0.6 and 0.75);
\draw [rotate around={-60:(0,-1)}] (-0.56,0.98) arc (180-18:180-330:0.6 and 0.75);
\draw [rotate around={60:(0,-1)}] (-0.56,0.98) arc (180-18:180-330:0.6 and 0.75);
\draw [rotate around={180:(0,-1)}] (-0.56,0.98) arc (180-18:180-330:0.6 and 0.75);
\draw (0,-1) circle (2.5);
\draw (0,-1) circle (1);

\node at (4.4,1) {when $n=3$ this is a solid};
\node at (4.65,0.5) {torus (collapsed) lying on $R^3$};
\node at (4.2,-1.1) {$S^{n-2} \times D^2$};
\draw [->] (3.2,-1.1) -- (2.6,-1.1);
\draw [decorate, decoration = {brace,amplitude=4pt}] (2.6,-1.5) -- (2.6,-3.3);
\node at (3.4,-2.4) {$\epsilon > 0$};

\node at (-7,0) {\hspace{1em}};
	\end{tikzpicture}
	\caption{Image of PL immersion of $\de M$.}
	\label{basic-model}
\end{figure}

To recall, the mapping cylinder of $f: X \twoheadrightarrow Y$, is $MC_f \coloneqq X \times [0,1] \indep Y \slash_{x \times 1 \sim x^\pr \times 1 \iff f(x) = f(x^\pr)}$. A mapping cylinder structure on a space $Z$ is a homeomorphism to $MC_f$. In our applications $\de Z = X$ and the homeomorphism will be PL.

\subsection*{Rigidity/Flexibility of Bing spines}
Subjects that are entirely governed by an $h$-principle (e.g. the theory of immersions with positive codimension) are less interesting than those (e.g. symplectic geometry) which have both rigid and flexible aspects. Thus we note with interest that Bing spines exhibit both rigid and flexible behaviors. 

We say two Bing spines are \emph{regularly homotopic} if there is a PL regular homotopy $h_t$, $0 \leq t \leq 1$, of $\de M$ so that for all $t$, $h_t(\de M)$ is a Bing spine of $M$. Recall that a regular homotopy is a locally 1-1 PL map $\de \times I \ra M \times I$ that restricted at each level $t$ is a locally 1-1 PL map $\de M \times t \ra M \times t$. We require that the entire homotopy is locally flat, as explained above, which is stronger than merely asking that each time slice is locally flat. It follows that $h_t(\de M \times I)$ is \emph{also} a spine for $M \times I$. Regular homotopy is quite a strong equivalence relation compared to the concordance-based equivalence relation for which we prove an $h$-principle. So we call what we prove a ``weak $h$-principle.'' Regular homotopy can also be called ``sliced concordance.''

As an example, we will explain in Section \ref{survey of examples} how a Bing spine for the solid torus obtained by ``pushing in with $k$ hands", for $k>0$, is regularly homotopic to one obtained by ``pushing in with $(k+1)$ hands". Figure \ref{basic-model} illustrates this example for $k=6$. In contrast, we will also give examples of Bing spines that are not regularly homotopic, in particular, the obvious Bing spine, $S^2 \times \frac{1}{2}$, of $S^2\times I$ is not regularly homotopic to any other Bing spines. This exhibits a rigid aspect of Bing spines.

\subsection*{Acknowledgement}
The second named author is thankful to be at the Max Planck Institut f\"{u}r Mathematik in Bonn. She is grateful for its hospitality and financial support. The first named author would like to thank the referee for perceptive and useful comments.

\section{Proof of Theorem \ref{approximation}}
The idea of the proof is to take as input a general sequence: $c_1,\dots,c_k$ of simplicial collapses whose end result is a spine $S$ for $M$, and identify in the sequence certain \emph{dangerous} collapses of $n$-simplices where the immerse property of the map from $\de M$ to the partially-collapsed-$M$ would be lost. A simple example of a dangerous collapse for a 2D square is arrow $b$, collapsing the triangle $c_2$, shown in Figure \ref{resolve-cones}. The collapse $c_2$ is dangerous because it leaves behind a free vertical edge and mapping of the $S^1$ boundary over that edge creates a $180^\circ$ bend, which is not an immersion. To fix this problem we subdivide the cell to be collapsed, and collapse all of it except a small Bing ruff drawn in Figure \ref{resolve-cones} as a pair of black triangles. This Bing ruff is a spine for a small $S^{n-2} \times D^2$ neighborhood in the simplex labeled $c_2$ of the boundary of the free face about to be collapsed. The artwork may be slightly misleading as literally our pair of black triangles is an $S^0 \times D^2$, which is top dimensional, whereas we do collapse further and leave only a well-controlled Bing spine of $S^{n-2} \times D^2$, the Bing ruff. This trick rescues the immersion property of $\de M$ as we cover the original collapse sequence with a more refined one.

\subsection*{Building Bing ruffs in $\mathbf{S^{n-2} \times D^2}$ in all dimensions $\mathbf{n\geq 3}$}
Figure \ref{basic-model} (bottom) shows (partial collapse towards) the local model $BS_{n,\epsilon}$ when $n=3$ and for $\epsilon \approx \frac{2\pi l}{6}$, $l$ a length scale, and 6 being the ``number of hands" in Figure \ref{basic-model}. If the solid torus of figure \ref{basic-model} is metrically a circle of radius $l$ cross a disk of radius $\epsilon$, the projection to the first factor is an $\epsilon$-map in that it moves points no more than $\epsilon$. Notice that if we pick the number of bands $\#_h \geq \frac{2\pi l}{\epsilon}$, the collapse to the $\#_n$-BS will also be an $O(\epsilon)$-map, actually an $O(\epsilon)$-immersion. The ability to control how far these collapse move points will be critical to the approximation aspect of Theorem \ref{approximation}: The number of hands must increase and the disk radius must decrease for these (locally flat PL) immersions to approach zero motion.

Next we construct the local model on $S^2 \times D^2$, i.e. the case $n=4$. For a first approximation, consider the suspension of Figure \ref{basic-model}. In a movie of the suspension, the radius of $S^1$ in Figure \ref{basic-model} (the core $S^1$ lying in $\R^2 \subset \R^3$) grows from zero to $l$ and back to zero. As it stands the model has non-immersive behavior at the two radius $=0$ suspension points, where it locally has the structure of a cone. These bad cone points can be resolved locally using a concrete choice for a Bing spine for the 4-ball (an example of which is   the model shown in Figure 3 of \cite{freedman-nguyen}). In detail, this is done by marking a small solid torus $T$ within a chart in $S^3$. Then in $T$ draw the mirror image of Figure \ref{basic-model} above and glue the 4D model in to resolve the cone points.

More precisely, let $f_{3,\epsilon}: S^1 \times D^2 \ra BS_{3,\epsilon}$ be the collapsing immersion and $F: S^1 \times D^2 \times [-1,+1] \hookrightarrow BS_{3,\epsilon} \times [-1,+1]$ be $f \times \id_{[-1,1]}$. Using a MC structure induced by $f_{3,\epsilon}$ we may taper $F$ near $\{-1,+1\}$ so it is still a (lcoally flat PL) immersion on $S^1 \times \de D^2 \times [-1,1]$ but so that its image $F(S^1 \times \de D^2) \times [-1,1]$ becomes $BS_{3,\epsilon} \times [-1,1] \cup S^1 \times D^2 \cup \{-1,1\}$ as draw in Figure \ref{collapsing-immersion}

\begin{figure}[ht]
	\centering
	\begin{tikzpicture}[scale=1.5]
		\draw (-1,0) arc(0:-180:1 and 0.3);
\draw[dashed] (-1,0) arc(0:180:1 and 0.3);
\draw (-1,0) -- (-1,2);
\draw (-2,2) ellipse (1 and 0.3);
\draw (-3,2) -- (-3,0);
\node at (-2,2) {$Y$};
\node at (-2,0) {$Y$};
\node at (-2.55,0.75) {$F$};
\draw[->] (-2.7,0.5) -- (-2.4,0.5);
\draw[->] (-2.7,1) -- (-2.4,1);
\draw[->] (-2.7,1.5) -- (-2.4,1.5);
\draw[->] (-1.3,0.5) -- (-1.6,0.5);
\draw[->] (-1.3,1) -- (-1.6,1);
\draw[->] (-1.3,1.5) -- (-1.6,1.5);
\draw[gray] (-2.06,-0.1) -- (-2.06,1.9);
\draw[gray] (-1.92,2.1) -- (-1.92,0.1);
\node at (-0.3,2.3) {image $F$};
\draw[->] (-0.9,2.3) -- (-1.8,1.9);
\node at (-2,-0.7) {$S^1 \times D^2 \times [-1,1], BS_{3,\epsilon} \times [-1,1]$};
\node at (-2,-1.1) {with $BS_{3,\epsilon}$ drawn as $Y$};

\node at (0.25,1) {$\longrightarrow$};

\draw (3.5,0) arc(0:-180:1 and 0.3);
\draw[dashed] (3.5,0) arc(0:180:1 and 0.3);
\draw (3.5,0) -- (3.5,2);
\draw[pattern = north east lines, pattern color = gray] (2.5,2) ellipse (1 and 0.3);
\path[pattern = north east lines, pattern color = gray] (2.5,0) ellipse (1 and 0.3);
\draw (1.5,2) -- (1.5,0);
\node at (2.5,2) {$Y$};
\node at (2.5,0) {$Y$};
\node at (1.95,1) {$F_\text{new}$};
\draw[->] (1.6,0.5) -- (2.2,0.4);
\draw[->] (1.6,0.7) -- (2.2,0.7);
\draw[->] (1.6,0.3) -- (1.9,0.1);
\draw[->] (1.6,1.3) -- (2.2,1.3);
\draw[->] (1.6,1.5) -- (2.2,1.6);
\draw[->] (1.6,1.7) -- (1.9,1.9);
\draw[->] (3.4,0.6) -- (2.8,0.5);
\draw[->] (3.4,0.8) -- (2.8,0.8);
\draw[->] (3.4,0.4) -- (3.1,0.2);
\draw[->] (3.4,1.2) -- (2.8,1.2);
\draw[->] (3.4,1.4) -- (2.8,1.5);
\draw[->] (3.4,1.6) -- (3.1,1.8);
\draw[gray] (4.5-2.06,-0.1) -- (4.5-2.06,1.9);
\draw[gray] (4.5-1.92,2.1) -- (4.5-1.92,0.1);
\node at (2.5,-0.7) {For $F_\text{new}$ the image includes};
\node at (2.5,-1.1) {$S^1 \times D^2 \times \{-1,1\}$};
	\end{tikzpicture}
	\caption{}
	\label{collapsing-immersion}
\end{figure}

Using $F_{\text{new}}$ on $S^1 \times D^2 \times [-1,1] \subset S^1 \times D^2 \times [-1,1] \cup B_+^4 \cup B_-^4 \cong S^2 \times D^2$ we obtain an immersion of $\de(S^2 \times D^2)$ onto $BS_{3,\epsilon} \times [-1,1] \cup \de B_+^4 \cup \de B_-^4$. Now compose this immersion with the immersion from $\de B_{\pm}^4$ to a BS for $B^4$. The composition of locally flat PL immersions is also a locally flat PL immersion, resolving the original singular points at the origins of $B_{\pm}^4$. (In our notation, we think of $S^2$ as $S^1 \times [-1,1] \cup D_+^2 \cup D_-^2$ and cross all terms with $D^2$ to obtain the isomorphism $S^1 \times D^2 \times [-1,1] \cup B_+^4 \cup B_-^4 \cong S^2 \times D^2$ used above.)

In a similar manner we may construct our local model on $S^{n-2} \times D^2$ from $S^{n-3} \times D^2$, first by suspending and then resolving the non-immersed suspension points by the local models of $(n-1)$-dimensional Bing houses given in \cite{freedman-nguyen}. This completes the first step in the proof of Theorem 1; the local models, or Bing ruffs, in $S^{n-2} \times D^2$ are now in hand, with the collapse immersions $f_{n,\epsilon}$ being $\epsilon$-maps, provided radius of $D^2$ is taken sufficiently small.

\subsection*{Using local models to deal with dangerous collapses}
Proceed through the collapse sequence until the first dangerous collapse is met at simplex $\sigma$ (say $c_2$ in Figure \ref{resolve-cones}). Near the boundary of this face (PL homeomorphic to $S^{n-2}$) identify an embedded $S^{n-2} \times D^2$, corresponding to the location of the two black triangles in Figure \ref{resolve-cones}. (Two triangles in the illustration because $S^0 \cong$ two points.) Now make a fine subdivision of $\sigma$ for which this $S^{n-2} \times D^2$ is a subcomplex and furthermore, $BS_{n,\epsilon} \subset S^{n-2} \times D^2$ is also an $(n-1)$-subcomplex. Now instead of collapsing $\sigma$ in one step by pushing in the free face of $\sigma$, collapse in many steps by first collapsing $\sigma$ to its subspace $S^{n-2} \times D^2$, and then proceeding to collapse $S^{n-2} \times D^2$ to the Bing spine $BS_{n,\epsilon}$. At this point, schematically we have reached the far end of the arrow labeled $\gamma$ in \ref{resolve-cones}. The reader should picture the boundary of the original square now wrapping immersively around the perimeter (after arrow $\gamma$) going up and around the vertical flag. Again note that the illustration may mislead as the two black triangles represent the $(n-1)$-dimensional BS and not the $n$-dimensional $S^{n-2} \times D^2$. At this point we have covered the first dangerous collapse with a long sequence of $\de$-immersive collapses the end result of which is an immersion from $\de M$ which is $\epsilon$-close to the original collapse sequence up to this point.

In the original collapse sequence there may now be cells of $\dim < n$ to be collapsed. As indicated by arrow $\delta$ (Figure \ref{resolve-cones}). We do these collapses, but not fully; we ``collapse'' these lower dimensional cells \emph{sufficiently} to so that after each ``collapse'' our immersion from $\de M$ remains $O(\epsilon)$ close to the original collapse sequence. A partial collapse is a simplicial homeomorphism approximating collapse. Concantenations of partial collapses also approximate concantenations of actual collapses.

When we next come, in the original sequence, to an $n$-dimension simplex with a free face to be collapsed we repeat the entire process. Keeping track of epsilons is trivial since once a Bing ruff is created (resulting in an $\epsilon$-approximation) it never participates in forming a subsequent Bing ruff (one ruff per dangerous collapse) so we do not need to arrange any inductive hierarchy of epsilons.

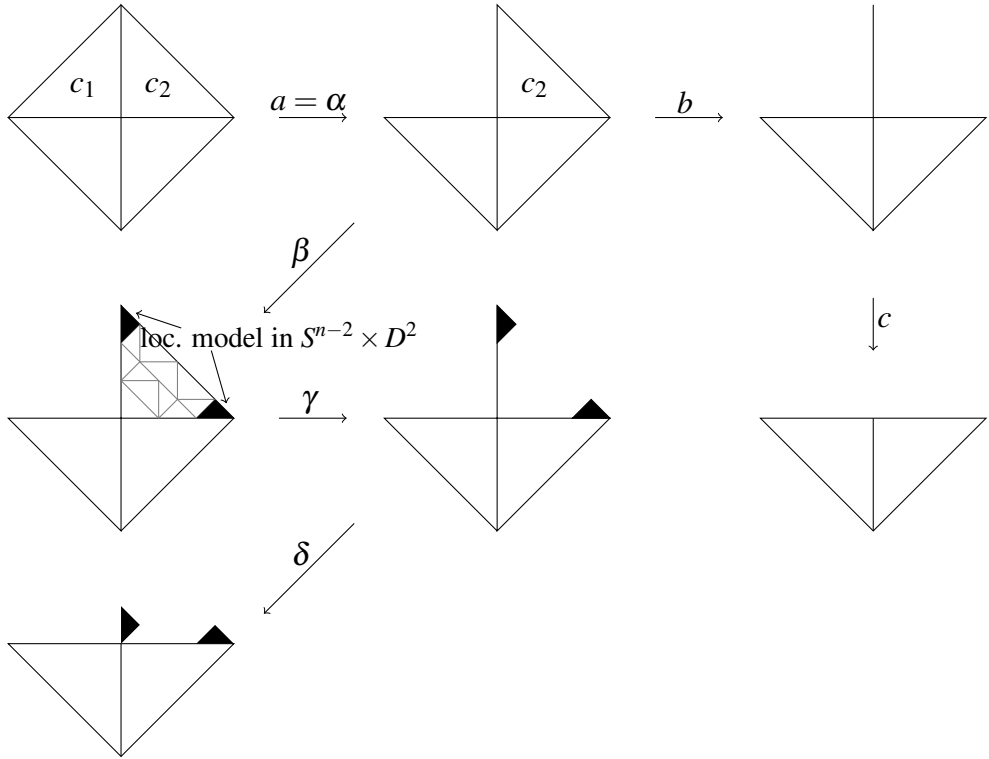
\begin{figure}[!ht]
	\centering
	\begin{tikzpicture}
		\draw (0,0) -- (3,0) -- (1.5,1.5) -- (1.5,-1.5) -- cycle;
		\draw (1.5,1.5) -- (0,0);
		\draw (1.5,-1.5) -- (3,0);
		\node at (1,0.4) {$c_1$};
		\node at (2,0.4) {$c_2$};
		\draw [->] (3.6,0) -- (4.5,0);
		\node at (4,0.2) {$a=\alpha$};

		\draw (5,0) -- (8,0) -- (6.5,1.5) -- (6.5,-1.5) -- cycle;
		\draw (6.5,-1.5) -- (8,0);
		\node at (7,0.4) {$c_2$};
		\draw [->] (8.6,0) -- (9.5,0);
		\node at (9,0.2) {$b$};
		\draw (11.5,1.5) -- (11.5,-1.5) -- (13,0) -- (10,0) -- (11.5,-1.5);
		\draw [->] (11.5,-2.4) -- (11.5,-3.1);
		\node at (11.65,-2.7) {$c$};

		\draw [->] (4.6,-1.4) -- (3.4,-2.6);
		\draw (0,-4) -- (3,-4) -- (1.5,-2.5) -- (1.5,-5.5) -- cycle;
		\draw (1.5,-5.5) -- (3,-4);
		\draw [color = gray] (1.5,-3.5) -- (2,-4);
		\draw [color=gray] (1.5,-3) -- (2.5,-4);
		\draw [color=gray] (1.75,-3.25) -- (1.5,-3.5) -- (2,-3.5) -- (2,-4) -- (2.25,-3.75);
		\draw [color=gray] (1.5,-3) -- (1.75,-2.75) -- (1.75,-3.25) -- (2.25,-3.25) -- (2.25,-3.75) -- (2.75,-3.75) -- (2.5,-4);
		\path [fill=black] (1.5,-2.5) -- (1.75,-2.75) -- (1.5,-3) -- cycle;
		\path [fill=black] (3,-4) -- (2.5,-4) -- (2.75,-3.75) -- cycle;
		\node at (3.9,-1.8) {$\beta$};
		\node at (3.6,-2.9) {\small{loc. model in $S^{n-2} \times D^2$}};
		\draw [->] (2.3,-2.8) -- (1.7,-2.6);
		\draw [->] (2.7,-3.1) -- (2.9,-3.8);
		\draw [->] (3.6,-4) -- (4.5,-4);
		\node at (4,-3.8) {$\gamma$};

		\draw (6.5,-2.5) -- (6.5,-5.5) -- (8, -4) -- (5,-4) -- (6.5,-5.5);
		\draw [fill=black] (6.5,-2.5) -- (6.75,-2.75) -- (6.5,-3) -- cycle;
		\draw [fill=black] (8,-4) -- (7.5,-4) -- (7.75,-3.75) -- cycle;

		\draw (10,-4) -- (13,-4) -- (11.5,-5.5) -- cycle;
		\draw (11.5,-4) -- (11.5,-5.5);
		\draw [->] (4.6,-5.4) -- (3.4,-6.6);
		\node at (3.9,-5.8) {$\delta$};

		\draw (0,-7) -- (3,-7) -- (1.5,-8.5) -- cycle;
		\draw (1.5,-8.5) -- (1.5,-7);
		\path [fill=black] (1.5,-6.5) -- (1.75,-6.75) -- (1.5,-7) -- cycle;
		\path [fill=black] (3,-7) -- (2.5,-7) -- (2.75,-6.75) -- cycle;
	\end{tikzpicture}
	\caption{The dangerous collapse $c_2$ (arrow $b$) is avoided by covering the latin sequence $a,b,c,\dots$ with the greek sequence $\alpha,\beta,\gamma,\delta,\dots$.}
	\label{resolve-cones}
\end{figure}

Covering the collapse sequence $c_1, \dots, c_k$ as shown in Figure \ref{resolve-cones} whenever a non-immersive (``dangerous") collapses is about to occur constitutes an algorithm for producing the approximation claimed in Theorem 1. The key point is that by allowing the immersion of $\de M$ to wrap around the Bing ruff instead of sharply folding (arrow b) the immersive property will be maintained during the covered collapse sequence. Also by picking $\epsilon$ sufficiently small, and making the ``partial collapses" close to their corresponding (full) collapses, the $\epsilon$-approximation condition will be fulfilled; the resulting BS consists of the original spines together with Bing ruffs attached to $S$ and within radius $\epsilon$ of their attaching region.
\qed

\section{Rigidity versus weak $h$-principle}
In this section we make some initial observations on this dichotomy, with the space of Bing spines BS$(M)$ for $M = S^2 \times [-1,1]$, $M = S^1 \times B^2$, and $M = B^3$ as examples, finding both rigid and flexible behavior. Then we proceed to establish a weak $h$-principle.

\subsection{A short survey of examples}\label{survey of examples}
Let us consider some examples of Bing spines for $B^3, S^1 \times B^2$ and $S^2 \times I$. In Figure \ref{examples}, examples $a_1, a_2, ...$ are Bing spines for $B^3$, examples $b_1, b_2, ...$ are Bing spines for $S^1\times B^2$, and $c_0, c_1, c_2, ...$ are Bing spines for $S^2\times I$. 

\begin{figure}[hp]
	\centering
	\begin{tikzpicture}
		\input{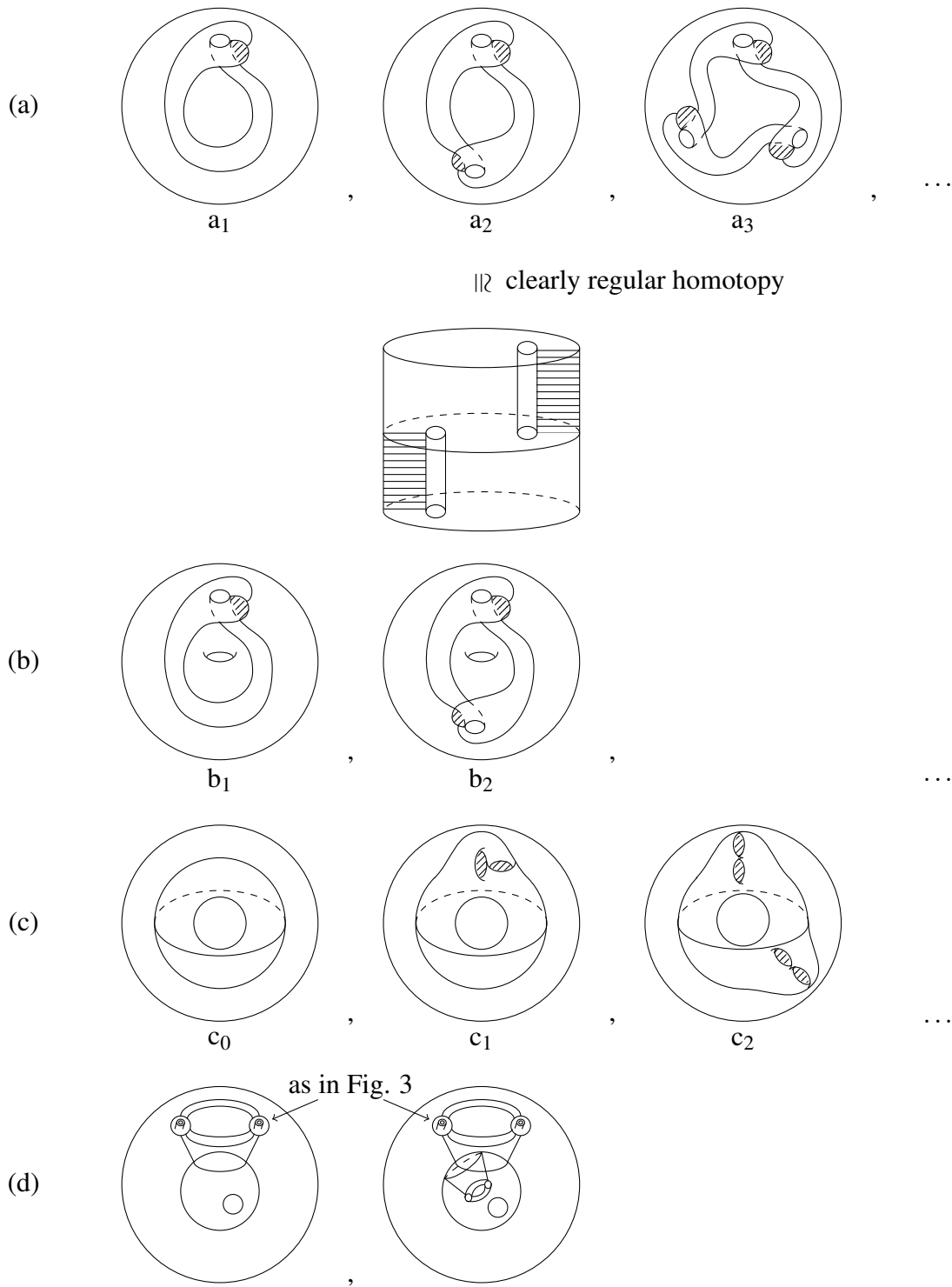}
	\end{tikzpicture}
	\caption{(a) Bing houses in $B^3$ (b) Bing spines in $S^1 \times B^2$ (c) Bing spines in $S^2 \times I$ (d) More Bing spines for $S^2 \times I$, which share with $c_0$ the property that their ``separating set'' (points of BS in the closure of both components of $S^2 \times I \setminus B$) is homeomorphic to $S^2$.}
	\label{examples}
\end{figure}

\begin{observation} Examples $a_1, a_2, a_3, \dots$ are all regularly homotopic, as are examples $b_1, b_2, b_3, \dots$. However, while $c_1, c_2, c_3, \dots$ are also regularly homotopic, we show below that $c_0$ is \emph{not} regularly homotopic to $c_i$ for any $i \geq 1$.
\end{observation}

The regular homotopies above are quite similar, they involve finding an appropriate handle slide of double point disks. A double point disk is a maximal disk in the BS over whose interior the immersion is precisely two to one. In Figure \ref{regular-homotopy} we show the regular homotopy from $c_2$ to $c_1$.

\begin{figure}[!ht]
	\centering
	\begin{tikzpicture}
		\draw (1,0) to (3,0) to [out=0, in=-90] (4,0.5) to (4,2) to [out=90,in=0] (3.2,3) to [out=180,in=0] (2,2.2) to [out=180,in=0] (0.8,3) to [out=180,in=90] (0,2) to (0,0.5) to [out=-90,in=180] (1,0);
\node at (2,-0.3) {$(c_2)$};
\draw [pattern = horizontal lines] (3.2,1.45) ellipse (0.25 and 0.3);
\draw (3.2,1.75) arc (-90:90:0.15 and 0.625);
\draw [dashed] (3.2,1.75) arc (-90:-270:0.15 and 0.625);
\path [pattern=north east lines] (3.2,2.375) ellipse (0.15 and 0.625);
\draw [pattern=horizontal lines] (0.8,1.45) ellipse (0.25 and 0.3);
\draw (0.8,1.75) arc (-90:90:0.15 and 0.625);
\draw [dashed] (0.8,1.75) arc (-90:-270:0.15 and 0.625);
\path [pattern=north west lines] (0.8,2.375) ellipse (0.15 and 0.625);
\draw [->] (4.5,1.5) -- (5.5,1.5);

\draw (7,0) to (9,0) to [out=0, in=-90] (10,0.5) to (10,2) to [out=90,in=0] (9.2,3) to [out=180,in=0] (8,2.2) to [out=180,in=0] (6.8,3) to [out=180,in=90] (6,2) to (6,0.5) to [out=-90,in=180] (7,0);
\draw [pattern = horizontal lines] (6.8,1.45) ellipse (0.25 and 0.3);
\draw (6.8,1.75) arc (-90:90:0.15 and 0.625);
\draw [dashed] (6.8,1.75) arc (-90:-270:0.15 and 0.625);
\path [pattern=north west lines] (6.8,2.375) ellipse (0.15 and 0.625);
\draw [pattern = horizontal lines] (9.2,1.45) ellipse (0.25 and 0.3);
\draw (7.4,2.4) to [out=90,in=135] (7.5,2.46) to [out=-45,in=90] (7.6,2.3) .. controls (7.6,0.9) and (9.2,2.6) .. (9.2,1.75);
\draw [dashed] (7.4,2.3) .. controls (7.4,-0.05) and (9.2,3.5) .. (9.2,1.75);
\path [pattern = north east lines] (7.4,2.4) to [out=90,in=135] (7.5,2.46) to [out=-45,in=90] (7.6,2.3) .. controls (7.6,0.9) and (9.2,2.6) .. (9.2,1.75) .. controls (9.2,3.5) and (7.4,-0.05) .. (7.4,2.3) -- cycle;

\draw (1,-4.5) to (3,-4.5) to [out=0, in=-90] (4,-4) to (4,-2.5) to [out=90,in=0] (3.2,-1.5) to [out=180,in=0] (2,-2.3) to [out=180,in=0] (0.8,-1.5) to [out=180,in=90] (0,-2.5) to (0,-4) to [out=-90,in=180] (1,-4.5);
\draw [pattern = horizontal lines] (0.8,1.7-4.75) ellipse (0.25 and 0.3);
\draw [pattern = horizontal lines] (3.2,1.7-4.75) ellipse (0.25 and 0.3);
\draw [dashed] (0.8,-2.75) arc (-90:-270:0.15 and 0.625);
\draw (0.8,-2.75) arc (-90:90:0.15 and 0.625);
\draw [->] (5.5,-0.5) -- (4.5,-1.5);
\draw (0.65,-2.1) .. controls (1.3,-3.5) and (2.9,-2.35) .. (3.1,-2.75);
\draw [dashed] (0.65,-2.1) .. controls (1.3,-1.9) and (1.1,-2.3) .. (1.5,-2.4) .. controls (2.3,-2.5) and (3.6,-1.9) .. (3.1,-2.75);
\path [pattern = north east lines] (0.65,-2.1) .. controls (1.3,-3.5) and (2.9,-2.35) .. (3.1,-2.75) .. controls (3.6,-1.9) and (2.3,-2.5) .. (1.5,-2.4) .. controls (1.1,-2.3) and (1.3,-1.9) .. (0.65,-2.1);
\path [pattern=north west lines] (0.8,2.375-4.5) ellipse (0.15 and 0.625);

\draw [->] (4.5,-3) -- (5.5,-3);
\draw (7,-4.5) to (9,-4.5) to [out=0, in=-90] (10,-4) to (10,-2.5) to [out=90,in=0] (9.2,-1.5) to [out=180,in=0] (8,2.2-4.5) to [out=180,in=0] (6.8,-1.5) to [out=180,in=90] (6,-2.5) to (6,-4) to [out=-90,in=180] (7,-4.5);
\draw [pattern = horizontal lines] (9.2,1.45-4.5) ellipse (0.25 and 0.3);
\draw [pattern = horizontal lines] (6.8,1.45-4.5) ellipse (0.25 and 0.3);
\draw [dashed] (6.8,1.75-4.5) arc (-90:-270:0.15 and 0.625);
\draw (6.8,1.75-4.5) arc (-90:90:0.15 and 0.625);
\path [pattern=north west lines] (6.8,2.375-4.5) ellipse (0.15 and 0.625);
\draw (8.95,-3) arc (0:-180:0.95 and 0.2);
\draw [dashed] (8.95,-3) arc (0:180:0.95 and 0.2);
\path [pattern = north east lines] (8,-3) ellipse (0.95 and 0.2);
\node at (7.1,-3.4) {$a$};
\node at (8.5,-3.4) {$z$};
\node at (9.5,-3.4) {$b$};
\draw [->] (5.9,-4.8) -- (4.3,-5.9);
\node at (6.9,-5.5) {\small{collapse horizontal disk}};

\draw (1,-9) to (3,-9) to [out=0, in=-90] (4,-8.5) to (4,-7) to [out=90,in=0] (3.2,-6) to [out=180,in=0] (2,-6.8) to [out=180,in=0] (0.8,-6) to [out=180,in=90] (0,-7) to (0,-8.5) to [out=-90,in=180] (1,-9);
\draw [dashed] (0.8,-2.75-4.5) arc (-90:-270:0.15 and 0.625);
\draw (0.8,-2.75-4.5) arc (-90:90:0.15 and 0.625);
\draw (0.8,-7.25) to [out=180,in=90] (0.5,-7.6) to [out=-90,in=180] (0.8,-7.9) to [out=0, in=180] (3.2,-7.25) to [out=0,in=90] (3.5,-7.6) to [out=-90,in=0] (3.2,-7.9) to [out=180, in=0] (0.8,-7.25);
\path [pattern = horizontal lines] (0.8,-7.25) to [out=180,in=90] (0.5,-7.6) to [out=-90,in=180] (0.8,-7.9) to [out=0, in=180] (3.2,-7.25) to [out=0,in=90] (3.5,-7.6) to [out=-90,in=0] (3.2,-7.9) to [out=180, in=0] (0.8,-7.25);
\draw [->] (4.5,-7.5) -- (5.5,-7.5);
\path [pattern=north west lines] (0.8,2.375-9) ellipse (0.15 and 0.625);

\draw (7,-9) to (9,-9) to [out=0, in=-90] (10,-8.5) to (10,-7) to [out=90,in=0] (9.2,-6) to [out=180,in=0] (8,2.2-9) to [out=180,in=0] (6.8,-6) to [out=180,in=90] (6,-7) to (6,-8.5) to [out=-90,in=180] (7,-9);
\path [pattern=north west lines] (6.8,2.375-9) ellipse (0.15 and 0.625);
\draw (6.8,1.75-9) arc (-90:90:0.15 and 0.625);
\draw [dashed] (6.8,1.75-9) arc (-90:-270:0.15 and 0.625);
\draw [pattern = horizontal lines] (8,-7.35) to [out=180,in=0] (6.8,-7.3) to [out=180, in=90] (6.5,-7.5) to [out=-90, in=180] (6.8,-7.7) to [out=0,in=180] (8,-7.65) to [out=0,in=180] (9.2,-7.7) to [out=0,in=-90] (9.5,-7.5) to [out=90,in=0] (9.2,-7.3) to [out=180,in=0] cycle;
\node at (7.2,-6.7) {$x$};
\node at (9.5,-7.8) {$y$};
\node at (8,-9.3) {$(c_1)$};
	\end{tikzpicture}
	\caption{}
	\label{regular-homotopy}
\end{figure}

Observe that regular homotopy does \emph{not} permit the collapsing of either disks $x$ or $y$ in the last panel of Figure \ref{regular-homotopy} (collapsing either $x$ or $y$ produces a non-immersed point at $x \cap y$ in one of the two sheets covering that point.) However, collapsing $z$ has no such effect; it joins two previously separated sheets at points $a$ and $b$, rather than joining a sheet to itself. No singularity is created.

\begin{figure}[hp]
	\centering
	\begin{tikzpicture}[scale=1.25]
		\input{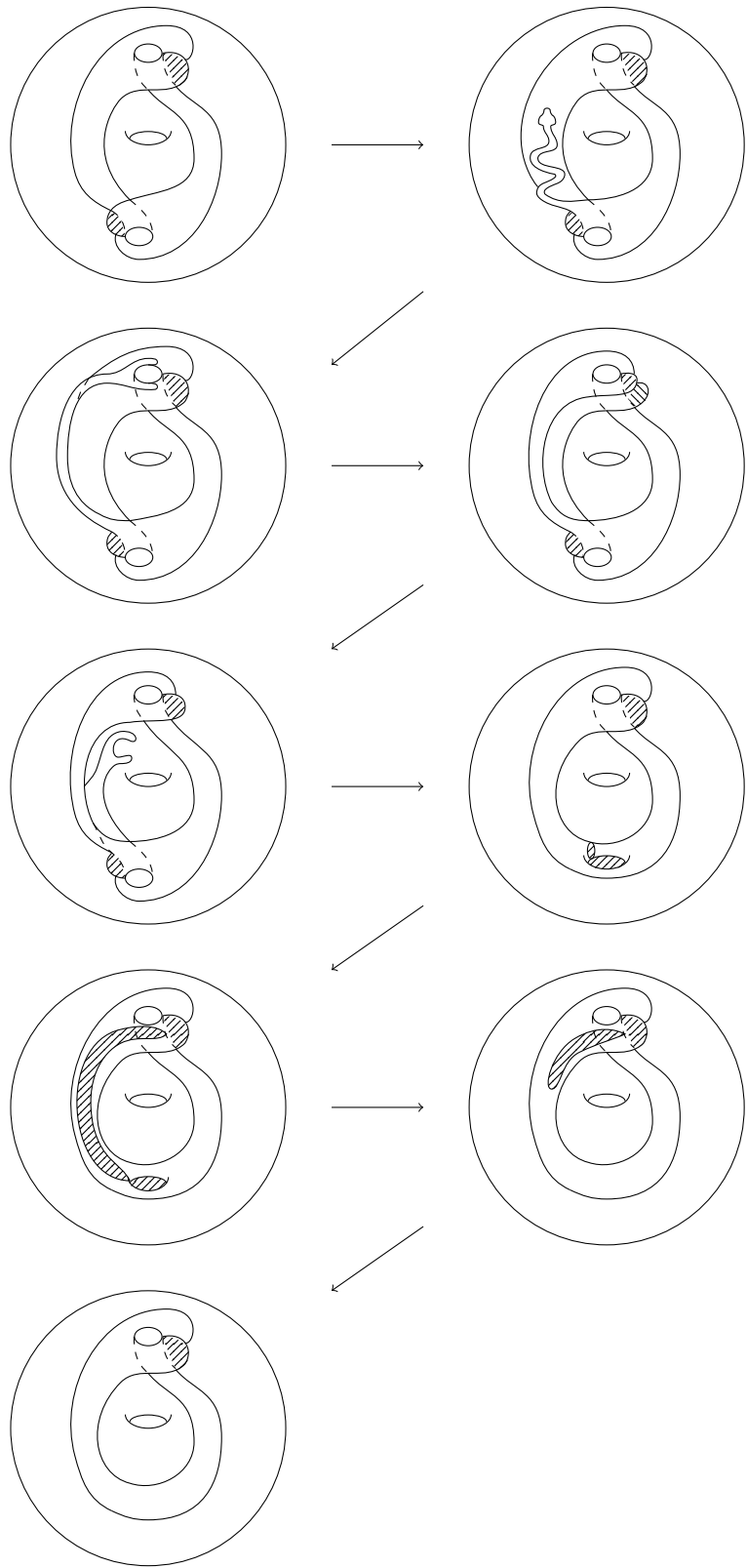}
	\end{tikzpicture}
	\caption{Regular homotopy from $b_2$ to $b_1$.}
\end{figure}

We conclude this short survey of examples with an example of rigidity among Bing spines: $c_0$ is not regularly homotopic to $c_i$, $i \leq 1$. Moreover, $c_0$ is completely rigid in the following sense.

\begin{proposition}
Any spine $c^\pr \subset S^2 \times I$ which is regularly homotopic to $c_0$ is in fact $h(c_0)$ for some PL homeomorphism $h: S^2 \times I \ra S^2 \times I$.
\end{proposition}

\begin{proof}
	Let $c^\pr_t$, $t \in [-\epsilon,\epsilon]$, be a Bing spine not of the form $h(c_0)$ for $t > 0$ but with $c_0^\pr = c_0$. Write $\gamma^\pr_t: S^2 \times \{-1,1\} \ra c_t^\pr$ for the function defining the mapping cylinder structure. Since $c_t^\pr \not\cong S^2$, $t > 0$, for all $t \geq 0$ there are a pair of distinct points $a_t, b_t \in S^2 \times \{+1\}$ or $a_t, b_t \in S^2 \times \{-1\}$ so that $\gamma_t(a_t) = \gamma_t(b_t)$. Let $a$ be an accumulation point of $\{a_t, t>0\}$ as $t \ra 0$ in $S^2 \times p$ for $p \in \{-1, 1\}$. The map $\Gamma: S^2 \times p \times [-\epsilon, \epsilon] \ra S^2 \times I \times [-\epsilon, \epsilon]$, defined by $\Gamma(s,p,t) = \gamma_t(s,p)$, is \emph{not} an immersion near $a$ as those pairs $(a_i, b_i)$ which lie in $S^2 \times p$ are, for each $i$, mapped together: $\Gamma$ is not locally injective near $a$.
\end{proof}

\begin{remark}
	The same argument shows that for all closed manifolds $M \cong N \times [-1,1]$, or more generally $M$ an $I$-bundle over $N$, the Bing spine $N \times 0$ is rigid.
\end{remark}

In fact, we conjecture that any Bing spine of $\Sigma^2 \times I$ that can be smoothly parametrized by an immersion of $\Sigma^2$ must be an embedded $\Sigma^2$. In this direction we prove:

\begin{proposition}\label{smooth Bing spine}
Let $M^3 = \Sigma\times [-1,1]$, where $\Sigma$ is a closed surface. If $B\subset M$ is a Bing spine parametrized by a smooth generic immersion $f\colon \Sigma \rightarrow M$, then $f$ is an embedding.
\end{proposition}

\begin{proof}
Since $f$ is generic, it has only standard double curve and triple point singularities. Let $T$ be the number of triple points. 

If $T=0$, then $f$ must also be without double curves since each double curve introduces a new free generator into the fundamental group of the image of $f$. Let $B$ be the image of $f$. Then $\pi_1(\Sigma) \cong \pi_1(B) \cong \pi_1 (\Sigma) \ast F_{\# \text{double curves}}$, so the number of double curves must be $0$ and $f$ is an embedding. 

So suppose that $T>0$. We may easily compute the Euler characteristic $\chi (B)$ as 
\[
	\chi (\Sigma) = \chi (B) = \chi (\Sigma) - \left(2T -\frac{1}{2}\cdot 2 \cdot 3 T\right) = \chi (\Sigma ) + T
\]
showing $T=0$. 

The correction $(2T -\frac{1}{2}\cdot 2 \cdot 3 T)$ above is accounted for as follows. Each triple point means the loss of two $0$-cells in the image $B$ of $f$. The pre-image graph of multiple points is $4$-valent so contains twice as many edges as vertices. It contains $3T$ vertices, so it contains $2\cdot 3T$ edges, and half of this total is missing from the cell structure of $B$.
\end{proof}

Similarly, one can show that if $M$ is a twisted $I$-bundle, then any Bing spine which is the image of a generic smooth immersion is actually an embedded surface.

We strongly conjecture that Proposition 2 holds without the genericity hypothesis, but were unable to find a proof. To appreciate the difficulty, note that any closed planar set can be the intersection of two sheets of a smooth immersion. It appears that a proof in complete generality would be a very nice exercise in planar topology which is certainly outside the scope of this paper, but a good problem for the interested reader.

On a related question, given a manifold $M$ with boundary $\partial M$ we ask if $M$ admits a Bing spine which can be parametrized by a smooth (not necessarily generic) immersion of $\partial M$ (or just some smooth manifold). Interval bundles over an $(n-1)$-manifold clearly have such BS. In the other direction, in \cite{freedman-nguyen} we proved that the $n$-ball $B^n$ do not have Bing spines that can be $C^1$ parametrized by a smooth manifold. The argument in \cite{freedman-nguyen} generalizes to give the following theorem that addresses the case when $\partial M$ is connected.

\begin{thm}
	Let $M^n$ be a smooth manifold of dimension $n \geq 2$ with connected boundary $\partial M$. If the intersection pairing between $H_{n-1}(M,\mathbb{Z}_2)$ and $H_1(M, \partial M, \mathbb{Z}_2)$ is trivial, then $M$ cannot have a smooth Bing spine $Y$, i.e. subset $Y \subset M$ which can be parameterized by a $C^1$-immersion $f \colon N \hookrightarrow Y \subset M$ for some closed, smooth $(n-1)$-manifold $N$, cannot have$M \setminus Y$ is homeomorphic to $\de M \times [0,\infty)$.
\end{thm}

\begin{proof}
The proof is essentially the same as the proof of Theorem A' in \cite{freedman-nguyen} where $\mathbb{D}^n$ is now replaced by $M^n$; Theorem 3 is a scholem to Theorem A'.
\end{proof}

The condition that ``the intersection pairing between $H_{n-1}(M,\mathbb{Z}_2)$ and $H_1(M, \partial M, \mathbb{Z}_2)$ is trivial" is necessary for the proof to work. This condition avoids spaces like twisted $I$-bundles (such as M\"obius bands), which have smooth Bing spines, namely the zero section. However, it is likely that this condition is not strictly necessary since there are plenty of $3$-manifolds with homology equivalent to twisted $I$-bundles. We suspect that many of these do not have $C^1$ Bing spines.

\subsection{Weak $h$-principle}\label{weak h-principle}
In this subsection we develop a small amount of semi-simplicial machinery and use it to give the \emph{precise} statement of Theorem \ref{roughthm} and its proof. We begin with:

\begin{remark}\label{concordantspines}
All spines for any $(M^n, \de M^n)$ are concordant (i.e. joined by some spine for $M \times I$). The proof is to relatively collapse $M \times I$ to a spine after a collapse has been fixed on $M \times \{\de I\}$. By applying Theorem \ref{approximation} in the relative context, all Bing spines are Bing-spine concordant. So concordance is a much weaker relation than regular homotopy. Concordance is our next topic.
\end{remark}

Given $(M^n, \de M^n)$, $n \geq 3$, let $\mathcal{S},\mathcal{S}_S,\mathcal{B}$, and $\mathcal{B}_S$ be four semi-simplicial spaces defined as follows.

$\mathcal{S}(M)$: vertices are spines for $M$, edges are spines for $M \times I$ which restrict to spines for $M \times -1$ and $M \times +1$, $\de I = \{\pm 1\}$, then inductively a $k$-simplex for $\mathcal{S}(M)$ is a spine for $M \times \Delta^k$ which restricts to each face of $\Delta_k$ to a $(k-1)$-simplex for $\mathcal{S}(M)$. $\mathcal{S}(M)$ is given the weakest topology making the inclusion of these simplices continuous.

$\mathcal{S}_S(M)$ is similarly defined except that it is based on \emph{sliced} concordance of spines rather than mere concordance as above. As previously remarked, sliced concordance agrees with regular homotopy. Thus a $k$-simplex of $\mathcal{S}_S(M)$ is a spine $S^k$ for $M \times \Delta^k$, so that for all $p \in \Delta^k$, $(S^k \cap \pi_2^{-1}(p))$ is a spine for $M$, where $\pi_2: M \times \Delta^k \ra \Delta^k$ is projection to the second factor.

$\mathcal{B}(M)$ and $\mathcal{B}_S(M)$ are defined exactly parallel to $\mathcal{S}(M)$ and $\mathcal{S}_S(M)$ but where all spines are now required to be Bing spines.

An immediate generalization of Remark \ref{concordantspines} says that the cone on any subcomplex $L$ of $\mathcal{S}(M)$ maps into $\mathcal{S}(M)$.
\begin{center}
	\begin{tikzpicture}
		\node at (0,0) {$L \hookrightarrow \mathcal{S}(M)$};
		\node at (-0.8,-1) {$\operatorname{cone}(L)$};
		\node[rotate=90] at (-0.8,-0.5) {$\hookleftarrow$};
		\draw[dashed,->] (0,-0.95) -- (0.5,-0.3);
	\end{tikzpicture}
\end{center}
Thus $\mathcal{S}(M)$ it is weakly contractible. If $\dim(M) \geq 3$, the same is true for $\mathcal{B}(M)$ by the approximation theorem in its relative form: if $M \times \Delta^k$ has a spine $S$ which is a BS when restricted to $M \times \de D^n$, then $S$ may be approximated, the Hausdorff topology, rel $M \times \de \Delta^k$ by a BS of $M \times \Delta^k$.

Since $M$ is assumed compact, sliced concordance in the Bing context $B_S$ is equivalent to building these semi-simplicial spaces from the regular homotopy relation. As we have seen, there is some rigidity in that context so that best we can do is prove an $h$-principle for $\mathcal{B}(M) \subset \mathcal{S}(M)$. This is Theorem \ref{roughthm}, now stated precisely.

\setcounter{thm}{1}
\begin{thm}[precise statement]\label{full-theorem}
	Assume $(M^n, \de M^n)$, $n \geq 3$, is a compact PL manifold with $\de M \neq \varnothing$. Then the semi-simplicial realizations $\mathcal{B}(M) \overset{\operatorname{inc}}{\hookrightarrow} \mathcal{S}(M)$ constructed above admit, for every $\epsilon > 0$, a strong left inverse: $\pi_\epsilon: \mathcal{S}(M) \ra \mathcal{B}(M)$, $\pi_\epsilon \circ \operatorname{inc} = \id_{\mathcal{B}(M)}$. $\pi_\epsilon$ being within $\epsilon$ of $\id_{\mathcal{S}(M)}$ in the sense that on a $k$-simplex $\sigma$ of $\mathcal{S}(M)$ the $k$-concordance $\sigma$ and $\pi_\epsilon(\sigma)$ are within $\epsilon$ in the Hausdorff topology on subsets of $M \times \Delta^k$. Furthermore there is a simplicial homotopy $\Pi_\epsilon$, within $\epsilon$ of the first factor projection, making $\operatorname{inc}: \mathcal{B}(M) \hookrightarrow \mathcal{S}(M)$ an $\epsilon$-controlled strong deformation retraction.
	\begin{center}
		\begin{tikzpicture}
			\node at (0,0) {$\mathcal{S}(M) \times 1 \hookrightarrow \mathcal{S}(M) \times I \hookleftarrow \mathcal{S}(M) \times 0$};
			\node at (-2.4,-1) {$\mathcal{B}(M)$};
			\node at (-1.1,-0.95) {$\overset{\operatorname{inc}}{\hookrightarrow}$};
			\node at (0.2,-1) {$\mathcal{S}(M)$};
			\draw[->] (-2.4,-0.3) -- (-2.4,-0.7);
			\draw[->] (0.2,-0.3) -- (0.2,-0.7);
			\node at (-2.2,-0.5) {\scriptsize{$\pi_\epsilon$}};
			\node at (0.45,-0.5) {\scriptsize{$\Pi_\epsilon$}};
			\draw[<-] (0.8,-1) -- (2.2,-0.3);
			\node at (1.65,-0.75) {\scriptsize{$\id$}};
		\end{tikzpicture}
	\end{center}
	where $\operatorname{dist}_{\mathrm{Haus}}(\Pi_\epsilon(s,t),s) \leq \epsilon$, $0 \leq t \leq 1$.
\end{thm}

\begin{proof}
	The construction of $\pi_\epsilon$ consists of specifying for every $S^i$ a spine of $M \times \Delta^i$ and every $\epsilon > 0$, an $\epsilon$-approximating $BS^i$. Given the collapse sequence $M \searrow S$, Theorem \ref{approximation}, in its relative form, provides such an approximating $BS^i$ on $M \times \Delta^i$. It is not necessary to have a canonical choice for $BS^i$, for each simplex $\Delta^i$ we may just choose some $BS^i$. The map $\pi_\epsilon$ can be specified to be the identity on $\mathcal{B}(M)$; if a spine is already a Bing spine do not modify it. The deformation $\Pi_\epsilon$ is obtained by applying the construction relatively to $\sigma \times I$, $\sigma$ being a spine on $M \times \Delta^i$. This completes the proof of Theorem \ref{full-theorem}.
\end{proof}

A partial result like Theorem \ref{roughthm} but for $B_S$ and $S_S$ would be of great interest.

\end{document}